\documentclass[11pt]{article}

\usepackage[utf8]{inputenc}
\usepackage{amsmath,amsthm,amssymb,enumerate,framed,mdwlist}
\usepackage{color,colortab}
\usepackage[square,numbers,sort&compress]{natbib}

\newcommand{\Q}{\mathbb Q}

\newcommand{\Z}{\mathbb{Z}}
\newcommand{\R}{\mathbb{R}}

\newcommand{\aff}{\operatorname{aff}}

\newcommand{\floor}[1]{\left\lfloor#1\right\rfloor}

\renewcommand{\epsilon}{\varepsilon}

\newtheoremstyle{mythmstyle}
	{\topsep}
	{\topsep}
	{\itshape}
	{}
	{\scshape}
	{.}
	{3pt}
	{}
\theoremstyle{mythmstyle}

\newtheorem{nn}{}[section]

\newtheorem{theorem}[nn]{Theorem}

\newtheorem{prop}[nn]{Proposition}

\newtheorem{claim}{Claim}

\newtheorem{REMARK}[nn]{Remark}
\newenvironment{remark}{\begin{REMARK}}{\end{REMARK}}

\newenvironment{cpf}{\begin{trivlist} \item[] {\em Proof of Claim.}}{\hspace*{\stretch{1}} $\diamond$ \end{trivlist}}

\numberwithin{equation}{section}

\allowdisplaybreaks

\begin{document}

\title{An extreme function which is nonnegative and discontinuous everywhere}

\author{Amitabh Basu\footnote{Department of Applied Mathematics and Statistics, The Johns Hopkins University. Supported by the NSF grant CMMI1452820.} \and Michele Conforti\footnote{Dipartimento di Matematica ``Tullio Levi-Civita'', Universit\`a degli Studi di Padova, Italy. Supported by the grants ``SID 2016'' and ``PRIN 2016''.}\and Marco Di Summa\footnotemark[2]}


\maketitle

\begin{abstract}
We consider Gomory and Johnson's infinite group model with a single row.  Valid inequalities for this model are expressed by valid functions and it has been recently shown that any valid function is dominated by some nonnegative valid function, modulo the affine hull of the model.
Within the set of nonnegative valid functions, extreme functions are the ones that cannot be expressed as convex combinations of two distinct valid functions. In this paper we construct an
extreme function $\pi:\R \to [0,1]$ whose graph is dense in $\R \times [0,1]$. Therefore $\pi$ is discontinuous everywhere.
\end{abstract}

\section{Introduction}
Given $b\in \R\setminus \Z$, Gomory's single-row infinite group model is defined as the set $I_b$ of {\em finite support} functions $y: \R \to \Z_+$ (i.e., $y$ takes value $0$ on all but a finite subset of $\R$) such that
$$\sum_{x \in \R} y(x)x \equiv b.$$
We use the symbol ``$\equiv$'' for the congruence modulo 1. This model arises as a relaxation of a single-row integer program in tableau form. We will use $\R^{(\R)}$ to denote the set of finite support functions from $\R$ to $\R$, and $\R^{(\R)}_+$ will denote the subset of these functions that are nonnegative.

Note that, as $b\not\in \Z$, the function $y_0$ that takes value 0 on all  of $\R$ is not in $I_b$.
We study halfspaces in the space of finite support functions $y: \R \to \R$ that contain $I_b$ but not $y_0$. Any such halfspace can be described by an inequality of the form
$$\sum_{x\in \R}\pi(x)y(x) \geq 1,$$
where the function $\pi:\R\to \R$ (which is not necessarily finite support) gives the coefficients of the inequality. We use the notation $H_\pi$ to denote this halfspace, and call $\pi$ a {\em valid function} whenever $I_b\subseteq H_\pi$.

Valid functions were introduced by Gomory and Johnson in~\cite{infinite,infinite2} and they have been the subject of intensive research, which is summarized in the recent surveys~\cite{basu2016light,basu2016light2,basu2015geometric}. Valid functions have been mostly studied under the nonnegativity assumption, i.e., $\pi\geq 0$. Recently Basu et al.~\cite{basu2017structure} have shown that this assumption is legitimate, as any valid function has an equivalent nonnegative form, see Theorem \ref{th:nonneg} below. The same theorem shows that valid functions that take negative values exist and most of them have graphs which are dense in the plane, which makes them intractable. We will elaborate this point in the next section. Up to now, it was believed that such a pathological behavior could not happen when the function is nonnegative. In this paper we construct an extreme function $\pi:\R \to [0,1]$ whose graph is dense in $\R \times [0,1]$. 
In the remainder we will assume that a valid function is nonnegative, unless explicitly stated.
\smallskip

A valid function $\pi:\R\to\R_+$  {\em dominates} the  valid function $\tilde\pi:\R\to\R_+$ if $\pi\ne\tilde\pi$ and $\pi\le\tilde\pi$ componentwise. Since $I_{b}\subseteq \R^{(\R)}_+$ and $\R^{(\R)}_+\cap H_{\pi}\subsetneq  \R^{(\R)}_+\cap H_{\tilde\pi}$ whenever $\pi$  dominates $\tilde\pi$, we can say that $\pi$  is a ``better'' function than $\tilde\pi$.

 A valid function $\pi:\R\to\R_+$ is {\em minimal} if it is not dominated by another valid function. An application of Zorn's lemma shows that every valid function that is not minimal is dominated by a minimal function, see e.g.~\cite[Theorem1]{basu-hildebrand-koeppe-molinaro:k+1-slope}.

Gomory and Johnson~\cite{infinite,infinite2} proved the following characterization of minimal functions.

\begin{theorem}\label{th:minimal}
A function $\pi:\R\to\R_+$ is  minimal if and only if:
\begin{itemize}
\item $\pi(x)+\pi(y)\ge\pi(x+y)$ for every $x,y\in \R$ (subadditivity);
\item $\pi(x) + \pi(b- x) = 1$ for every $x\in \R$  (symmetry);
\item $\pi(z) = 0$ for every $z\in \Z$.
\end{itemize}
\end{theorem}

The above theorem implies $\pi(b)=1$ and $\pi(x)=\pi(x+z)$ for every $x\in \R$ and $z\in \Z$  (periodicity). Therefore it suffices to define a minimal function on the interval $[0,1)$.\smallskip

A valid function $\pi:\R\to\R_+$ is \emph{extreme} if $\pi_1=\pi_2$ for every pair of valid functions $\pi_1,\pi_2:\R \to\R_+$ such that $\pi=\frac{1}{2}\pi_1+\frac{1}{2}\pi_2$. It is well known that every extreme  function is minimal. 

\paragraph{``Complicated'' extreme functions.}

The most well known extreme function is the (Gomory mixed-integer) GMI function: if we assume $0<b<1$, the GMI function is
$$\pi(x)=\begin{cases}
\frac{x}{b}&\mbox{if $0\le x \le b$}\\
\frac{1-x}{1-b}&\mbox{if $b\le x<1$}.
\end{cases}$$

The GMI function is piecewise linear and has two slopes.
\smallskip

In 2003, Gomory and Johnson \cite{tspace} conjectured that extreme functions that are continuous  are always  piecewise linear functions.  Basu et al. \cite{bccz08222222} constructed a sequence of extreme functions that are piecewise linear, have two slopes,  but an increasing number of breakpoints. The limiting function of such a sequence is an extreme function that is not piecewise linear. This is a counterexample to the above conjecture.

Gomory and Johnson constructed an extreme function that is piecewise linear with 3 slopes. It appears to be hard to construct extreme functions that are piecewise linear with many slopes. Indeed, all known families of piecewise linear extreme functions had at most 4 slopes until 2013 when  Hildebrand, in an unpublished result, constructed an extreme function that is piecewise linear with 5 slopes. More recently K\"oppe and Zhou~\cite{koeppe-zhou:extreme-search} constructed an extreme function that is piecewise linear with 28 slopes and Basu et al.~\cite{basu2016extreme} constructed a sequence of extreme functions that are piecewise linear, with strictly increasing number of slopes. The limiting function of such a sequence is an extreme function with an infinite number of slopes.
\smallskip

 There are several constructions of nonnegative functions that are extreme and discontinuous. However, to the best of our knowledge, these functions are piecewise linear with a finite number of discontinuities all of which are jump discontinuities, with  a  notable exception: K\"oppe and Zhou   \cite[p.~338]{koeppe-zhou-ipco2017}   constructed a function which is nonnegative, extreme and discontinuous at all points in two subintervals of the $[0,1]$ interval.
In this paper, we construct a function that is nonnegative, extreme and whose graph is dense in $\R \times [0, 1]$. So it is discontinuous everywhere and the left and right limits do not exist at any point. The function in~\cite{koeppe-zhou-ipco2017} also has no left and right limits at the points of discontinuity. The graph of this function is not dense in $\R \times [0, 1]$; in fact, the closure of this graph  is the union of finitely many piecewise linear curves in $\R^2$.
We refer to~\cite{koppe2016equivariant}  for a discussion of discontinuities in extreme functions.

\section{The affine hull of $I_b$, Hamel bases and the nonnegative form of a valid function}

We summarize here some results that appear in~\cite{basu2017structure} and are useful for our construction.

A function $\theta:\R\to\R$ is {\em additive} if it satisfies the following \emph{Cauchy functional equation}:
\begin{equation}\label{eq:additive}
\theta(u+v)=\theta(u)+\theta(v)\mbox{ for all }u,v\in\R.
\end{equation}

Equation \eqref{eq:additive} has been extensively studied, see e.g.~\cite{aczel1989functional}.
Given any $c\in\R$, the linear function $\theta(x)=cx$ is obviously a solution to \eqref{eq:additive}. However, there are other solutions that we describe below.

A {\em Hamel basis} for $\R$ is a basis of  $\R$ over the field $\Q$. In other words, a Hamel basis is a subset $\mathcal B\subset\R$ such that, for every $x\in\R$, there exists a {\em unique} choice of a  finite subset $\{ a_1,\dots, a_t\}\subseteq \mathcal B$  and nonzero {\em rational} numbers $\lambda_1,\dots,\lambda_t$ such that
\begin{equation}\label{eq:combination}
x=\sum_{i=1}^t\lambda_i a_i.
\end{equation}
The existence of a Hamel basis $\mathcal B$ is guaranteed under the axiom of choice.

For every $ a\in \mathcal B$, let $c( a)$ be a real number. Define $\theta$ as follows: for every $x\in\R^n$, if \eqref{eq:combination} is the unique decomposition of $x$, set
\begin{equation}\label{eq:theta}
\theta(x)=\sum_{i=1}^t \lambda_ic( a_i).
\end{equation}
It is easy to check that a function of this type is additive and the following theorem proves that all additive functions are of this form~\cite[Theorem 10]{aczel1989functional}.

\begin{theorem}\label{Th-Hamel}
Let $\mathcal B$ be a Hamel basis of $\R^n$. Then every additive function is of the form \eqref{eq:theta} for some choice of real numbers $c( a),\, a\in \mathcal B$.
Furthermore the graph of an additive function that is not a linear function is dense in $\R\times \R$.
\end{theorem}

The following result is an immediate extension of a result of Basu, Hildebrand and K\"oppe (see \cite[Propositions 2.2--2.3]{basu2016light}) and appears in~\cite{basu2017structure}.
We define the affine hull of $I_b$ as the smallest affine subspace of $\R^{(\R)}$ that contains $I_b$.

\begin{prop}\label{prop:affine-hull}
The affine hull of $I_b$  is described by the equations
\begin{equation}\label{eq:aff-Ib}
\sum_{p\in\R^n}\theta(p)y(p)=\theta(b)
\end{equation}
for all additive functions $\theta:\R\to\R$ such that $\theta(p)=0$ for every $p\in\Q^n$.
\end{prop}

\begin{remark}  Assume $b\in \Q$ and let $\theta:\R\to\R$ be an additive function such that $\theta(p)=0$ for every $p\in\Q^n$. Let  $\pi:\R\to\R_+$ be a valid function which is piecewise linear, such as the GMI function.  The above proposition  shows that the functions $\pi$ and $\pi+\theta$ are equivalent in the sense that
$$\aff(I_b)\cap H_{\pi}=\aff(I_b)\cap H_{\pi+\theta}.$$
By Theorem \ref{Th-Hamel}, the graph of $\pi+\theta$   is dense in $\R\times \R$.
\end{remark}

Basu et al.~\cite{basu2017structure} show the following:

\begin{theorem}\label{th:nonneg}
 Assume $b\in \Q$. For every valid function $\pi$ for $I_b$, there exists an additive function $\theta$ such that  $\theta(p)=0$ for every $p\in\Q^n$ and the valid function  $\pi':=\pi+\theta$ satisfies $\pi'\ge0$.
\end{theorem}

 A similar theorem holds without the assumption $b\in \Q$, see \cite{basu2017structure}.

\section{The construction}


Fix $b=1/2$.
Let $\mathcal B$ be a Hamel basis of $\R$ such that $b\in\mathcal B$.
Then every $x\in\R$ can be uniquely written in the form
\begin{equation*}\label{eq:x}
x=\lambda_b^x b+\sum_{a\in \mathcal A^x}\lambda_a^x a,
\end{equation*}
where $\mathcal A^x$ is a finite subset of $\mathcal B$ and the coefficients $\lambda_b^x,\lambda_a^x$, $a \in \mathcal{A}^x$ are all rational.

Define the following function $\pi:\R\to[0,1]$:
\[\pi(x):=
\begin{cases}
\lambda^x_b-\floor{\lambda^x_b} & \mbox{if $\lambda^x_b$ is not an odd integer,}\\
1 & \mbox{if $\lambda^x_b$ is an odd integer.}
\end{cases}\]
Note that $\pi(x)\equiv\lambda^x_b$ for every $x\in\R$.

\begin{prop}
Function $\pi$ is minimal for $I_b$.
\end{prop}

\begin{proof}
We show that $\pi$ satisfies the conditions of Theorem \ref{th:minimal}. It is clear that $\pi$ is nonnegative. Also, if $x\in\Z$ then $\pi(x)=0$ because $\lambda^x_b$ is an even integer.\smallskip

We now verify that $\pi(x)+\pi(b-x)=1$ for every $x\in\R$.
Note that $\lambda^{b-x}_b=1-\lambda^x_b$.
If $\lambda^x_b\in\Z$, then $\lambda^{b-x}_b\in\Z$, and $\lambda_b,\lambda^{b-x}_b$ have opposite parity. It follows that in this case exactly one of $\pi(x),\pi(b-x)$ takes value 0 and the other takes value 1, thus $\pi(x)+\pi(b-x)=1$.
Therefore we now assume $\lambda^x_b\notin\Z$. Since $\pi(x)\equiv\lambda^x_b$ and $\pi(b-x)\equiv \lambda^{b-x}_b=1-\lambda^x_b$, we have $\pi(x)+\pi(b-x)\in\Z$. Since $0<\pi(x)<1$ and $0<\pi(b-x)<1$ (as $\pi$ is bounded between 0 and 1, and $\lambda^x_b\notin\Z$), we necessarily have $\pi(x)+\pi(b-x)=1$.\smallskip

We finally show that $\pi(x)+\pi(y)\ge\pi(x+y)$ for every $x,y\in\R$.
Note that $\lambda_b^{x+y}=\lambda_b^x+\lambda_b^y$.
If $\lambda_b^x$ and $\lambda_b^y$ are both even integers, then so is $\lambda_b^{x+y}$. Then in this case $\pi(x)=\pi(y)=\pi(x+y)=0$ and therefore the inequality $\pi(x)+\pi(y)\ge\pi(x+y)$ is satisfied.
Thus we now assume that $\lambda_b^x$ or $\lambda_b^y$ is not an even integer. Note that in this case $\pi(x)+\pi(y)>0$.
Since $\pi(x)\equiv\lambda_b^x$ and $\pi(y)\equiv\lambda_b^y$, we have $\pi(x)+\pi(y)\equiv\lambda_b^x+\lambda_b^y=\lambda_b^{x+y}\equiv\pi(x+y)$. Thus $\pi(x)+\pi(y)\equiv\pi(x+y)$. Since $\pi(x)+\pi(y)>0$ and $\pi(x+y)\le1$, we conclude that $\pi(x)+\pi(y)\ge\pi(x+y)$.
\end{proof}

\begin{prop}
Function $\pi$ is extreme for $I_b$.
\end{prop}

\begin{proof}
Assume that $\pi=\frac12\pi_1+\frac12\pi_2$. As $\pi$ is minimal, $\pi_1,\pi_2$ are also minimal functions for $I_b$. We show that $\pi=\pi_1=\pi_2$.

\begin{claim}
$\pi_i(x)=\pi_i(\lambda_b^xb)$ for every $x\in\R$ and $i=1,2$.
\end{claim}

\begin{cpf}
Take any $x\in\R$ and define $y:=\lambda_b^xb-x$. Note that $\lambda^y_b=0$, $x+y=\lambda_b^xb$, and $\lambda_b^{x+y}=\lambda_b^x$.
As $\lambda^y_b=0$, we have $\pi(y)=0$. Since $\pi(y)=\frac12\pi_1(y)+\frac12\pi_2(y)$ and $\pi_1,\pi_2$ are nonnegative functions, this implies that $\pi_1(y)=\pi_2(y)=0$.
Moreover, $\pi(x)=\pi(x+y)$, as $\lambda_b^x=\lambda_b^{x+y}$; so, $\pi(x) + \pi(y) = \pi(x+y)$. By subadditivity of $\pi_1,\pi_2$, this implies that $\pi_i(x) +\pi_i(y)=\pi_i(x+y)=\pi_i(\lambda_b^xb)$ for $i=1,2$. Since $\pi_i(y) = 0$, we obtain the desired result.\end{cpf}

We now analyze some cases to prove that $\pi(x)=\pi_1(x)=\pi_2(x)$ for every $x\in\R$. Because of the above claim and since we also have $\pi(x)=\pi(\lambda_b^xb)$ for every $x\in\R$, we can assume $x\in\Q$, i.e., $x=\lambda^x_bb$. Moreover, since $\pi,\pi_1,\pi_2$ are all periodic modulo $\Z$, we can assume $0\le x<1$, i.e., $0\le\lambda_b^x<2$.

\begin{enumerate}
\item
Assume first $\lambda^x_b=1/q$ for some positive integer $q$. By minimality of $\pi_i$, $q\pi_i(x)\ge\pi_i(qx)=\pi_i(q\lambda_b^{x}b)=\pi_i(b)=1$ for $i=1,2$.
Moreover, $q\pi(x)=q(1/q)=1$. Since $\pi(x)=\frac12\pi_1(x)+\frac12\pi_2(x)$, it follows that $q\pi_i(x)=1$ for $i=1,2$. This shows that $\pi(x)=\pi_1(x)=\pi_2(x)=1/q$.
\item
Assume now $\lambda^x_b=p/q$ for some integers $p$ and $q$ such that $q>0$ and $0\le p\le q$.
By subadditivity of $\pi_i$, $\pi_i(x)=\pi_i(\lambda^x_b b)\le p\pi_i(b/q)=p\pi(b/q)=p/q$ for $i=1,2$, where the second equation follows from case 1.
Moreover, $\pi(x)=p/q$ because $0\le\lambda^x_b\le1$.
This shows that $\pi(x)=\pi_1(x)=\pi_2(x)=p/q$.
\item
Assume now $\lambda^x_b=p/q$ for some positive integers $p$ and $q$ such that $q< p\le3q/2$.
Define $y:=3b-2x$. Since $\lambda_b^y=(3q-2p)/q$, $y$ satisfies the assumptions of case 2, thus $\pi(y)=\pi_1(y)=\pi_2(y)=(3q-2p)/q.$
By minimality of $\pi_i$, $\pi_i(x)+\pi_i(y)\ge\pi_i(x+y)=\pi_i(3b-x)=\pi_i(b-x)=1-\pi_i(x)$ for $i=1,2$, which implies $\pi_i(x)\ge p/q-1$.
Since $\pi(x)=p/q-1$, we obtain $\pi(x)=\pi_1(x)=\pi_2(x)=p/q-1$.
\item
Finally, assume $\lambda^x_b=p/q$ for some positive integers $p$ and $q$ such that $3q/2<p<2q$.
In this case $3b-x$ satisfies the assumption of case 3. By using this along with symmetry and periodicity of $\pi_i$ and $\pi$, we obtain $\pi_i(x)=1-\pi_i(3b-x)=1-\pi(3b-x)=\pi(x)$ for $i=1,2$.
\end{enumerate}
\end{proof}

\begin{prop}
The graph of $\pi$ is dense in $\R\times[0,1]$ and $\pi$ is discontinuous everywhere.
\end{prop}

\begin{proof}
Let $f:\R\to\R$ be the unique additive function satisfying $f(b)=1$ and $f(a)=0$ for every $a\in\mathcal B$.
In other words, $f(x)=\lambda_b^x$ for every $x\in\R$.
Since $f$ is an additive function that is not linear, by Theorem \ref{Th-Hamel}, the graph of $f$  is dense in $\R\times\R$.

Since $f(x)=\lambda_b^x$ for every $x\in\R$, we have
\[\pi(x)=
\begin{cases}
f(x)-\floor{f(x)} & \mbox{if $f(x)$ is not an odd integer,}\\
1 & \mbox{if $f(x)$ is an odd integer.}
\end{cases}\]
Then the graph of $\pi$ contains all the points of the graph of $f$ that lie in $\R\times[0,1]$, and therefore it is dense in $\R\times[0,1]$. This also implies that $\pi$ is discontinuous everywhere.
\end{proof}

\providecommand\CheckAccent[1]{\accent20 #1}

\end{document}